\documentclass[10pt]{article}
\usepackage{amssymb,amsfonts,latexsym,wasysym}

\newtheorem{theorem}{Theorem}[section]
\newtheorem{remark}[theorem]{Remark}
\newtheorem{lemma}[theorem]{Lemma}

\newtheorem{prop}[theorem]{Proposition}
\newtheorem{cor}[theorem]{Corollary}

\renewcommand{\c}{\mathfrak{C}}
\newcommand{\tp}{\mathrm{tp}}

\newcommand{\dcl}{\mathrm{dcl}}
\newcommand{\dcli}{\mathrm{dcl}^{\mathrm{eq}}}
\newcommand{\acli}{\mathrm{acl}^{\mathrm{eq}}}
\newcommand{\g}[1]{\mathfrak{#1}}
\newcommand{\e}[1]{#1^\mathrm{eq}}

\def\Ind#1#2{#1\setbox0=\hbox{$#1x$}\kern\wd0\hbox to 0pt{\hss$#1\mid$\hss}
\lower.9\ht0\hbox to 0pt{\hss$#1\smile$\hss}\kern\wd0}
\def\ind{\mathop{\mathpalette\Ind{}}}
\def\Notind#1#2{#1\setbox0=\hbox{$#1x$}\kern\wd0\hbox to 0pt{\mathchardef
\nn=12854\hss$#1\nn$\kern1.4\wd0\hss}\hbox to
0pt{\hss$#1\mid$\hss}\lower.9\ht0 \hbox to
0pt{\hss$#1\smile$\hss}\kern\wd0}
\def\nind{\mathop{\mathpalette\Notind{}}}

\newenvironment{proof}{\vspace{-0.25cm}
{\bf Proof}: }{\hfill $\Box$}

\begin{document}
\title{Stable forking and imaginaries}
\author{Enrique Casanovas\thanks{Partially  supported by the Spanish government grant MTM 2011-26840 and  the Catalan government grant 2014SGR-437.} and Joris Potier\\
University of Barcelona}
\date{August 7, 2015}
\maketitle

\begin{abstract}  We prove that a theory $T$ has stable forking if and only if $\e{T}$  has stable forking.
\end{abstract}

\section{Introduction}

We follow the standard conventions, $T$ is a complete theory of language $L$ and $\c$ is its  monster model.  A formula $\varphi(x,y)$ (where $x,y$ are disjoint tuples of variables)  is stable if there are not $(a_i\mid i<\omega)$ and $(b_i\mid i<\omega)$  such that $\models \varphi(a_i,b_j)$ if and only if $i<j$. 

It is said that $T$ has stable forking if whenever a type $p(x)\in S(B)$  forks over some subset  $A\subseteq B$, there is some stable formula $\varphi(x,y)\in L$ and some tuple $b\in B$ such that $\varphi(x,b)\in p(x)$  and $\varphi(x,b)$ forks over $A$. The stable forking conjecture says that every simple theory has stable forking.

\begin{remark}\label{0}
\begin{enumerate}
\item If $\varphi(x,y)$  is a boolean combination of stable formulas $\varphi_i(x_i,y_i)$   (where $x_i\subseteq x$,  $y_i\subseteq y$  and $x_i\cap y_j=\emptyset$)  then $\varphi(x,y)$ is stable.
\item If $\varphi(x,y)$ is stable, then $\varphi^{-1}(y,x)  =\varphi(x,y)$   (the same formula with the role of $x$, $y$  interchanged)  is stable.
\item If $\varphi(x,a)\equiv \psi(x,b)$  and $\psi(x,z)$ is stable, then for some $\mu(y)\in \tp(a)$,  $\varphi(x,y)\wedge \mu(y)$  is stable.
\item In order to check that $T$ has stable forking, it is enough to consider types over models.
\item If $\varphi(x,y)$ is stable and $p(x)\in S_\varphi(M)$, then   $p(x)$  is definable by a boolean combination of formulas $\varphi(m,y)$  for some tuples $m\in M$.  The canonical base of $p(x)$ is an imaginary $e$, the canonical parameter of (any) definition of $p(x)$ over $M$.  If $A\subseteq M$, then $p(x)$  forks over $A$ if and only if  $e\not\in\acli(A)$.  For any model $N\supseteq M$, $p(x)$ has a unique $e$-definable extension $p^\prime(x)\in S_\varphi(N)$.
\end{enumerate}
\end{remark}    
\begin{proof}  For \emph{1}, \emph{2} and \emph{5}  see  chapters 6 and 8 of~\cite{Casanovas09} or    chapter 1 of~\cite{Pill96b}. For \emph{3} and \emph{4}  see~\cite{KimPillay01}.
\end{proof}

If $\varphi(x,y)$ is stable and $p(x)$ is a $\varphi$-type over a model $M$, then $p$ does not fork over its canonical base, an imaginary $e\in \dcli(M)$. Therefore, if $T$ has stable forking then for every type $p(x)\in S(M)$ there is a subset $A\subseteq M$ such that $|A|\leq |T|$ and $p$ does not fork over $A$. This means that if $T$ has stable forking, then $T$ is simple.  More generally, A. Chernikov has shown (see  Proposition 4.14 in~\cite{Chernikov12}) that if $T$ has simple forking (meaning that forking is always witnessed by a simple formula) then $T$ is simple.

There is not much progress on the stable forking conjecture.  B.~Kim proved in~\cite{Kim99}  that simple one-based theories with elimination of hyperimagnaries have stable forking.   A.~Peretz in~\cite{Peretz06}  proved that types of  SU-rank two in $\omega$-categorical supersimple theories have stable forking. In~\cite{PalacinWagner13}  D.~Palac\'{\i}n and F.O.~Wagner have shown that supersimple CM-trivial $\omega$-categorical theories have stable forking.   Finally, let us mention that stable forking implies weak elimination of hyperimaginaries (see~\cite{KimPillay01}).

We will need the following lemma on algebraic quantification of a stable formula:

\begin{lemma}\label{1}  If $\varphi(x,y)\in L$  is stable and $\theta(v,x)\vdash \exists^{=n}x\theta(v,x)$, then $\psi(v,y)= \exists x(\theta(v,x)\wedge \varphi(x,y))$  is stable.
\end{lemma}
\begin{proof}  Assume  $\models \psi(a_i,b_j)$  iff  $i<j$.    For each $i<\omega$,   $\models \exists x \theta(a_i,x)$  and hence there  are different $c_i^1,\ldots,c_i^n$   such that  $\models \theta(a_i,c_i^k)$  for all $k=1,\ldots, n$.   Whenever  $i<j<\omega$   choose  some  $k_{ij}$  such that  $1\leq k_{ij}\leq n$  and $\models \theta(a_i, c_i^{k_{ij}})\wedge \varphi(c_i^{k_{ij}},b_j)$. By Ramsey's theorem, for some infinite $I\subseteq\omega$ there is some $k$  such that   $1\leq k\leq n$  and $\models \theta(a_i, c_i^{k})\wedge \varphi(c_i^{k},b_j)$   for  all  $i,j\in I$ such that $i<j$.  Then   for $i,j\in I$: $\models \varphi(c^k_i,b_j) $  iff  $i<j$,  which shows that $\varphi(x,y)$ is unstable. 
\end{proof}

The following remark is a stronger version of item 4 of Remark~\ref{0}, with a similar proof. We won't use it in this article.  The proof uses generalized $\varphi$-types (see chapter 6 of~\cite{Casanovas09}). The main point is that nonforking is transitive for these types (if $\varphi$ is stable) and over models they coincide with ordinary $\varphi$-types. The generalized $\varphi$-type of $a$ over $A$ is the set of all formulas in $\tp(a/A)$  which are equivalent to boolean combinations of $\varphi$-formulas over the monster model.

\begin{remark} If whenever a  type $p(x)\in S(N)$  forks over an elementary submodel $M\subseteq N$, there is an instance of a stable formula in $p(x)$  witnessing forking over $M$, then $T$ has stable forking.  
\end{remark}
\begin{proof} We can assume that $T$ is simple. Assume $A\subseteq B$  and $a\nind_A B$. Choose a model $M\supseteq A$  such that  $M\ind_A Ba$  and note that $a\nind_M B$.   Now choose a model $N\supseteq MB$  such that  $N\ind_{MB}a$  and note that  $a\nind_M N$.  By the assumption, there is a stable formula $\varphi(x,y)\in L$ and some tuple $n\in N$  such that $\models\varphi(a,n)$ and $\varphi(x,n)$  forks over $M$. Let $p(x)$  be the $\varphi$-type of $a$ over $N$  and $q(x)$ the generalized $\varphi$-type of $a$ over $B$.  Since $p(x)$  forks over $A$ but it does not fork over $B$,  $q(x)$  forks over $A$.  Hence there is some formula $\psi(x,z)\in L$ and some tuple $b\in B$  such that $\psi(x,b)$ is in $q(x)$ and forks over $A$.  By item 3 of Remark~\ref{0}, there is some formula $\mu(z)\in \tp(b)$  such that $\psi^\prime(x,z)=\psi(x,z)\wedge \mu(z)$ is stable. But $\models\psi^\prime(a,b)$  and $\psi^\prime(x,b)$ forks over $A$.
\end{proof}

\section{Main result}

In this section $T$ is a simple theory, but since stable forking implies simplicity, in fact it is not necessary to add this assumption to the propositions below. Hence forking and dividing is the same thing in $T$.

\begin{prop}\label{3}
If $T$ has stable forking, then $\e{T}$ has stable forking over real parameters.
\end{prop}
\begin{proof} Let $A\subseteq B\subseteq \c$, let $e\in \e{\c}$ be an imaginary of sort $E$ and assume  $e\nind_A B$. Choose a $|A|^++\omega$-saturated model $M\supseteq B$ such that $e\in\dcli(M)$ and a representative $a$ of $e=a_E$  such that $a\ind_e M$. It follows that $a\nind_A B$. By assumption, there is some stable formula $\delta(x,y)\in L$ and some tuple $b\in B$  such that $\models \delta(a,b)$  and  $\delta(x,b)$  forks over $A$.  Consider the type $p(x)=\tp_\delta(a/M)$.   It has a definition $d_px\delta(x,y)$  which is a boolean combination of formulas of the form $\delta(m,y)$  for some tuples $m\in M$.  The definition is, therefore, an instance of a stable formula.  Note that $p(x)$  does not fork over $e$ and hence  its canonical basis  $d_F\in\e{M}$  is  in $\acli(e)$.  For some   $\chi(w,y)\in \e{L}$,  $\chi(d_F,y)$  defines $p(x)$.    Since  $d_px\delta(x,y)\equiv \chi(d_F,y)$, by item 3 of Lemma~\ref{0} for some $\mu(w)\in \tp(d_F)$, the formula $\chi(w,y)\wedge \mu(y)$  is stable.  Without loss of generality then $\chi(w,y)$ is stable.  Note that since $\delta(x,b)\in p(x)$,  $\models \chi(d_F,b)$.

\emph{Claim 1}:  If $q(w)=\tp(d_F)$, then $q(w)\cup\{\chi(w,b)\}$  forks over $A$.

\noindent
\emph{Proof}:  Assume not.  We will prove that $\delta(x,b)$ does not divide over $A$, which is a contradiction.  Let $(b_i\mid i<\omega)$ be an $A$-indiscernible sequence of tuples $b_i\equiv_A b$ and let us check that $\{\delta(x,b_i)\mid i<\omega\}$ is consistent.  By the saturation of $M$, we may assume that $b_i\in M$  for all  $i<\omega$. By our assumption in the proof, $q(w)\cup\{\chi(w,b_i)\mid i<\omega\}$ is consistent and hence we can find some realization $d^\prime_F\in\e{M}$ of this set of formulas.  Since $d_F\equiv d^\prime_F$, there is some  sequence $(b^\prime_i\mid i<\omega)$ in $\e{M}$  such that    $d_F (b^\prime_i\mid i<\omega)\equiv d^\prime_F (b_i\mid i<\omega)$.  Then  $\models \chi(d_F,b^\prime_i)$ for every $i<\omega$, which implies $\delta(x,b^\prime_i)\in p(x)$  and  $\models \delta(a,b^\prime_i)$  for all $i<\omega$.  Since $\{\delta(x,b^\prime_i)\mid i<\omega\}$ is consistent, $\{\delta(x,b_i)\mid i<\omega\}$ is consistent too.\\

With Claim 1 we can now choose some $\mu(w)\in q(w)$  such that  $\chi(w,b)\wedge \mu(w)$  forks over $A$.  Note that $\chi^\prime(w,y)= \chi(w,y)\wedge \mu(w)$  is stable.  Since $\chi^\prime(w,b)$ divides over $A$, this can be witnessed by an $A$-indiscernible sequence $(b_i\mid i<\omega)$ with $b_i\equiv_A b$  and some $k<\omega$  for which $\{\chi^\prime(w,b_i)\mid i<\omega\}$ is  $k$-inconsistent.  Since $d_F\in \acli(e)$, there is  some formula $\theta(v,w)\in\e{L}$  and some $n<\omega$  such that  $\models\theta (e,d_F)$  and $\theta(v,w)\vdash \exists^{=n}w\theta(v,w)$. Let $\varphi(v,y) =\exists w(\theta(v,w)\wedge \chi^\prime(w,y))$.    By Lemma~\ref{1}, $\varphi(v,y)$ is stable.  Since  $\models \varphi(e,b)$, it only remains to check that $\varphi(v,b)$  forks over $A$. This is done in the next claim.\\

\emph{Claim 2}: $\varphi(v,b)$  divides over $A$ with respect to  $l= n(k-1)+1$, witnessed by  $(b_i\mid i<\omega)$.

\noindent
\emph{Proof}:  Otherwise $\{\varphi(v,b_i)\mid i<\omega\}$  is consistent  and it is realized by some $e^\prime\in\e{\c}$.  For each $i<l$ choose some $d^i_F$  such that $\models  \theta(e^\prime,d^i_F)\wedge \chi^\prime(d^i_F, b_i)$.   The number of $d^i_F$  is  $\leq n$ and therefore, by choice of $l$, the mapping  $i\mapsto d^i_F$  has some fiber of cardinality $\geq k$. This shows that $\{\chi^\prime(w,b_i)\mid i<\omega\}$  is $k$-consistent, a contradiction with the choice of $k$.
\end{proof} 

\begin{prop}\label{4} If $\e{T}$  has stable forking over real parameters, then $\e{T}$ has stable forking.
\end{prop}
\begin{proof} By item 4 of Remark~\ref{0}, it is enough to consider types over models.  Assume $e\nind_A \e{M}$, where $M\subseteq \c$ is a model, $\e{M}=\dcli(M)$ is the corresponding imaginary model, $A\subseteq \e{M}$  and $e\in \e{\c }$.  Choose a set $A^\prime$ of representatives of the elements of $A$  such that $A^\prime\ind_A Me$. Then $e\nind_{A^\prime} M$ and by the assumption there is some stable formula  $\delta(v,y)\in\e{L}$  and some tuple $a\in A^\prime M$  such that   $\models \delta(e,a)$  and $\delta(v,a)$  forks over $A^\prime$. Let $p(v)=\tp_\delta(e/M)$  and let $c$ be its canonical base.  Since $e\ind_M A^\prime$, the unique global $\delta$-type  $\g{p}(v)\supseteq p(v)$  which is definable over $c$   extends  $\tp_\delta(e/A^\prime M)$.  Since $c$ is the canonical base of $\g{p}$ and $\g{p}$ forks over $A^\prime$, $c\not\in\acli(A^\prime)$. It follows that $c\not\in \acli(A)$.  Hence  $p (v)$    forks over $A$.  Let  $\varphi(v,b)$  be a finite conjunction of formulas of $p(v)$ which forks over $A$.  Since $\varphi(v,b)$  is  a  conjunction of $\delta$-formulas, it is an instance of a stable formula.  Moreover, $\models\varphi(e,b)$.
\end{proof}

\begin{cor}\label{5} $T$ has stable forking if and only if $\e{T}$ has stable forking.
\end{cor}
\begin{proof}  One direction follows from propositions~\ref{3} and~\ref{4}.  The rest is clear since $\e{L}$-formulas with real free variables are equivalent to $L$-formulas.
\end{proof}

\section{An example and some open problems}

We describe a theory $T$.  Its language  contains two binary relation  symbols $E,F$, both  are being interpreted as equivalence relations on the universe with some specific cross-cutting.  The equivalence relation $E$  has infinitely many classes, all infinite. On the other hand  $F$  has exactly one class of size $n$  for every $n\geq 1$, say consisting in the elements $a^n_1,\ldots, a^n_n$.  For each  $k\geq 1$, the elements  $a_k^k, a_k^{k+1},\ldots$  build an $E$-class.  With these specifications, the  set $\{a_k^n\mid  1\leq k\leq n<\omega\}$  is the universe of a model $M$ of $T$.   Note that $\dcl(\emptyset)=M$.   The formula  $E(x,y)$  has non finite cover property  and  $F(x,y)$  is  stable.  But  $\exists y(E(x,y)\wedge F(y,z))$  is unstable, as witnessed by the sequences  $(a^i_i\mid  i\geq 1)$ and $(a^j_1\mid j\geq 1)$.  This answers a question of M. C. Laskowski:  Lemma~\ref{1} can not be generalized  to the case where $\theta(v,x)$ is a non finite cover property formula.  On the other hand, it shows that the proof of Proposition~\ref{3} can not be carried out trying to prove that  the formula $\exists x( \pi_E(x)= v\wedge \delta(x,y))$  is stable  (where $\pi_E$ is the mapping sending each tuple to its $E$-equivalence class).  Since $T$ is interpretable in Presburger arithmetic, it   is  dp-minimal. But $T$    has the strict order property, hence it is not simple. This can be checked observing that the  $E$-class of $a_{11}$ is infinite and has a definable linear ordering.  It would be interesting to find a similar example in a simple theory $T$.

A. Chernikov  has raised  the question of whether Corollary~\ref{5}  can be generalized to  dependent forking.  See~\cite{Chernikov12}  for the relevant definitions.

Let us  finally mention a connected question asked by M. Ziegler.  Assume all $1$-types in $T$ have stable forking. Does it follow that $T$ has stable forking?  A  positive answer would be very helpful.

\bibliographystyle{abbrv}

\end{document}